\documentclass[11pt,reqno]{amsart}

\usepackage[T1]{fontenc}
\usepackage[utf8]{inputenc}
\usepackage{amsmath,amssymb,amsthm}
\usepackage[letterpaper,margin=0.75in]{geometry}
\usepackage{graphicx}
\graphicspath{{../images/}}
\usepackage{booktabs}
\usepackage[hidelinks]{hyperref}

\DeclareMathOperator{\vol}{vol}
\DeclareMathOperator{\conv}{conv}

\newcommand{\repobase}{https://github.com/csabamedgyeswork/inscribed-polytopes}
\newcommand{\repocommit}{maxvol-d4-v1}
\newcommand{\repohref}{\href{\repobase}{\texttt{\repobase}}}
\newcommand{\artpath}[1]{\href{\repobase/tree/\repocommit/#1}{\nolinkurl{#1}}}

\newcounter{artifact}
\newenvironment{artifactlist}
  {\sloppy 
   \begin{list}{[R\arabic{artifact}]}
     {\usecounter{artifact}%
      \setlength{\labelwidth}{3.1em}\setlength{\leftmargin}{3.7em}%
      \setlength{\labelsep}{0.6em}\setlength{\itemsep}{0.15ex}%
      \setlength{\parsep}{0pt}}}
  {\end{list}}
\newcommand{\artifact}[3]{\item\label{art:#1}\artpath{#2}, #3}
\newcommand{\artref}[1]{\hyperref[art:#1]{[R\ref*{art:#1}]}}

\newtheorem{theorem}{Theorem}[section]
\newtheorem{proposition}[theorem]{Proposition}
\newtheorem{lemma}[theorem]{Lemma}
\newtheorem{corollary}[theorem]{Corollary}

\theoremstyle{definition}
\newtheorem{definition}[theorem]{Definition}
\newtheorem{remark}[theorem]{Remark}

\newcommand{\citedname}{}
\newtheorem*{citedstatement}{\citedname}
\newenvironment{cited}[2]
  {\renewcommand{\citedname}{#1}\begin{citedstatement}[#2]}
  {\end{citedstatement}}
\title[Inscribed $4$-polytopes with seven and eight vertices]
      {Inscribed $4$-polytopes of maximal volume\\
       with seven and eight vertices}
\author{Csaba Medgyes}
\email{csaba.medgyes.work@gmail.com}

\begin{document}

\begin{abstract}
  We prove, with sum-of-squares certificates checked in exact arithmetic,
  that every cyclic $4$-polytope with seven vertices inscribed in
  $\mathbb{S}^{3}$ has volume at most $2/5$. Hence no $4$-polytope of maximal
  volume with seven vertices inscribed in $\mathbb{S}^{3}$ is cyclic, answering
  a question of Horv\'ath and L\'angi for $d = 4$, and with their theorem the
  maximal volume is $\sqrt{3}/4$. With eight vertices, the same method bounds
  the volume of every simplicial type other than the cross-polytope by
  $13/20$, so the regular cross-polytope, of volume $2/3$, is up to
  orthogonal transformations the unique maximiser.
\end{abstract}

\maketitle

\section{Introduction}
\label{sec:intro}

Which convex polytope with a given number of vertices, inscribed in a sphere,
has the largest volume? Fejes T\'{o}th raised this question in three
dimensions \cite{fejes-toth-1964}, and it remains open there for all but a few
vertex counts. Berman and Hanes found a local necessary condition for
optimality and used it to determine the optimal polyhedra with at most eight
vertices \cite{berman-hanes-1970}; Hoehner and Ledford recently settled the
case of nine vertices in a preprint \cite{hoehner-ledford-2026}. Beyond nine
vertices only numerical results are known: Mutoh reported a computer search
for up to $30$ vertices \cite{mutoh-2003}, and Hardin, Sloane and Smith
tabulate putatively optimal configurations for up to $130$
\cite{hardin-sloane-smith}.

In higher dimensions even less is known. In dimension $d$, Horv\'{a}th and
L\'{a}ngi determined the optimal polytopes with $d+2$ vertices, and, when $d$
is odd, with $d+3$ vertices \cite{horvath-langi-2016}. When $d$ is even, their
argument for $d+3$ vertices covers only polytopes that are not cyclic, and
they asked whether an optimal polytope can be cyclic
\cite[Problem~2]{horvath-langi-2016}; Horv\'{a}th restates the question in
\cite[Problem~2.5]{horvath-2016-survey}.

We answer this question in the first open case, $d=4$, where the polytopes
have seven vertices. We show that every cyclic polytope with seven vertices
inscribed in the unit sphere of $\mathbb{R}^4$ has volume at most $2/5$
(Theorem~\ref{thm:main}). A non-cyclic polytope attains the larger volume
$\sqrt{3}/4$, so no optimal polytope is cyclic (Corollary~\ref{cor:problem2}).
Combined with the result of Horv\'{a}th and L\'{a}ngi on non-cyclic
polytopes, this shows that the maximal volume is $\sqrt{3}/4$
(Corollary~\ref{cor:seven}).

We then settle the next case in four dimensions, eight vertices. Here the
natural candidate is the regular cross-polytope, the four-dimensional analogue
of the regular octahedron, which is optimal among polyhedra with six vertices
\cite{berman-hanes-1970}. Horv\'{a}th and L\'{a}ngi showed that in every
dimension the regular cross-polytope is optimal among polytopes with the
combinatorics of a cross-polytope \cite{horvath-langi-2016}, but this leaves
open whether a polytope of another combinatorial type does better. Since an
optimal polytope has only simplices as facets \cite{horvath-langi-2016}, it
suffices to consider simplicial polytopes. There are $37$ combinatorial types
of simplicial polytopes with eight vertices in four dimensions
\cite{grunbaum-sreedharan-1967}, and Firsching realised each of them with
rational vertices on the sphere \cite{firsching-2017}. We show that each of
the $36$ types other than the cross-polytope has volume at most $13/20$
(Theorem~\ref{thm:eight}), which is less than the volume $2/3$ of the regular
cross-polytope. Hence the regular cross-polytope is the optimal polytope with
eight vertices (Corollary~\ref{cor:eight}).

Both results rest on the same method. For a fixed combinatorial type we write
the volume as a polynomial in the coordinates of the vertices and bound it by
a sum-of-squares certificate \cite{parrilo-2003}, an algebraic identity that a
computer checks in exact rational arithmetic. The bound holds for every
configuration of points on the sphere, whether or not the points are in convex
position, so convexity never has to be verified.

Section~\ref{sec:prelim} collects the background, Section~\ref{sec:method}
describes the certificates and their verification, and
Sections~\ref{sec:results} and~\ref{sec:eight} prove the results for seven and
eight vertices.

\section{Preliminaries}
\label{sec:prelim}

\begin{cited}{Notation}{{\cite[Section~1]{horvath-langi-2016}}}
  Throughout the paper, by a polytope we mean a convex polytope. The vertex set
  of a polytope $P$ is denoted by $V(P)$. We denote the family of
  $d$-dimensional polytopes, with $n$ vertices and inscribed in the unit sphere
  $\mathbb{S}^{d-1}$, by $\mathcal{P}_{d}(n)$. We denote $d$-dimensional
  volume by $\vol_{d}$, and set
  $v_{d}(n) = \max\{\vol_{d}(P) : P \in \mathcal{P}_{d}(n)\}$. Note that by
  compactness, $v_{d}(n)$ exists for any value of $d$ and $n$.
\end{cited}

\begin{remark}
  \label{rem:attained}
  In more detail, for $n \geq d+1$: the volume of $\conv\{p_{1}, \ldots, p_{n}\}$
  is continuous on the compact set $(\mathbb{S}^{d-1})^{n}$, so it attains its
  maximum at some $p$, and this maximum is positive. Every point of
  $\mathbb{S}^{d-1}$ other than the $p_{i}$ lies outside their convex hull,
  so if two $p_{i}$ coincided, replacing one of them by such a point would
  strictly increase the volume. Hence the $p_{i}$ are distinct, they are the
  vertices of $\conv\{p_{1}, \ldots, p_{n}\} \in \mathcal{P}_{d}(n)$, and its
  volume is $v_{d}(n)$.
\end{remark}

\begin{definition}[{\cite[Lecture~2]{ziegler-1995}}]
  A \emph{facet} of a $d$-polytope is a face of dimension $d-1$, and a
  \emph{ridge} is a face of dimension $d-2$. A polytope is \emph{simplicial} if
  every facet is a simplex. Two simplicial polytopes are
  \emph{combinatorially equivalent} if a bijection between their vertex sets
  carries the vertex sets of the facets of one onto those of the other; an
  equivalence class is a \emph{combinatorial type}.
\end{definition}

\subsection{Property Z}

\begin{definition}[{\cite[Definition~1]{horvath-langi-2016}}]
  \label{def:propertyz}
  Let $P \in \mathcal{P}_{d}(n)$ be a $d$-polytope with
  $V(P) = \{p_{1}, p_{2}, \ldots, p_{n}\}$. If for each $i$, there is an open
  set $U_{i} \subset \mathbb{S}^{d-1}$ such that $p_{i} \in U_{i}$, and for any
  $q \in U_{i}$, we have
  \begin{equation*}
    \vol_{d}\bigl(\conv((V(P) \setminus \{p_{i}\}) \cup \{q\})\bigr)
      \leq \vol_{d}(P),
  \end{equation*}
  then we say that $P$ satisfies Property Z.
\end{definition}

\begin{remark}
  \label{rem:propertyz}
  Property Z is a local form of maximality: moving a single vertex slightly on
  the sphere does not increase the volume. Berman and Hanes
  \cite{berman-hanes-1970} introduced it for $d = 3$ and used, as we do,
  that every $P \in \mathcal{P}_{d}(n)$ with $\vol_{d}(P) = v_{d}(n)$
  satisfies Property Z: for $q$ outside $V(P) \setminus \{p_{i}\}$ the new hull has $n$
  points of $\mathbb{S}^{d-1}$ as vertices, so it lies in $\mathcal{P}_{d}(n)$
  or has dimension less than $d$.
\end{remark}

\begin{lemma}[{\cite[Lemma~1]{horvath-langi-2016}}]
  \label{lem:simplicial}
  Every $P \in \mathcal{P}_{d}(n)$ satisfying Property Z is simplicial.
\end{lemma}

\subsection{Polytopes with \texorpdfstring{$d+3$}{d+3} vertices}

\begin{theorem}[{\cite[Theorem~3]{horvath-langi-2016}}]
  \label{thm:hl}
  Let $P \in \mathcal{P}_{d}(d+3)$ satisfy Property Z. If $d$ is even, assume
  that $P$ is not cyclic. Then $P = \conv\{P_{1} \cup P_{2} \cup P_{3}\}$,
  where $P_{1}$, $P_{2}$ and $P_{3}$ are regular simplices inscribed in
  $\mathbb{S}^{d-1}$ and contained in three mutually orthogonal linear
  subspaces of $\mathbb{R}^{d}$. Furthermore:
  \begin{itemize}
    \item If $d$ is odd and $P$ has maximal volume over
      $\mathcal{P}_{d}(d+3)$, then the dimensions of $P_{1}$, $P_{2}$ and
      $P_{3}$ are $\lfloor d/3 \rfloor$ or $\lceil d/3 \rceil$. In particular,
      in this case we have
      \begin{equation*}
        (v_{d}(d+3) =) \vol_{d}(P) = \frac{1}{d!} \cdot \prod_{i=1}^{3}
          (k_{i}+1)^{\frac{k_{i}+1}{2}} k_{i}^{-\frac{k_{i}}{2}},
      \end{equation*}
      where $k_{1} + k_{2} + k_{3} = d$ and for every $i$, we have
      $k_{i} \in \{\lfloor \frac{d}{3} \rfloor, \lceil \frac{d}{3} \rceil\}$.
    \item The same holds for the dimensions of $P_{1}$, $P_{2}$ and $P_{3}$, if
      $d$ is even, and $P$ has maximal volume over the family of not cyclic
      elements of $\mathcal{P}_{d}(d+3)$, satisfying Property Z.
  \end{itemize}
\end{theorem}

\begin{remark}[{\cite[Remark~8]{horvath-langi-2016}}]
  \label{rem:dim4}
  For $d = 4$ the dimensions are $1$, $1$ and $2$, and the formula gives the
  volume
  ${\frac{1}{24} \cdot 2 \cdot 2 \cdot \frac{3\sqrt{3}}{2} = \frac{\sqrt{3}}{4}}$
  \artref{testprelim}. It is attained by the convex hull of a regular triangle
  and two diameters of $\mathbb{S}^{3}$ in mutually orthogonal linear
  subspaces.
\end{remark}

\subsection{Cyclic polytopes}

\begin{definition}[{\cite[Section~4]{horvath-langi-2016}}]
  A $d$-polytope with $n$ vertices is cyclic, if it is combinatorially
  equivalent to the convex hull of $n$ points on the moment curve
  $\gamma(t) = (t, t^{2}, \ldots, t^{d})$, $t \in \mathbb{R}$.
\end{definition}

\begin{theorem}[{Gale's evenness condition \cite{gale-1963},
    \cite[Theorem~0.7]{ziegler-1995}}]
  \label{thm:gale}
  Let $n > d \geq 2$ and $t_{1} < \cdots < t_{n}$. Then
  $\conv\{\gamma(t_{1}), \ldots, \gamma(t_{n})\}$ is a simplicial $d$-polytope,
  and a $d$-subset $S \subseteq \{1, \ldots, n\}$ indexes the vertices of a facet
  if and only if, for all $i < j$ not in $S$, the number of $k \in S$ with
  $i < k < j$ is even.
\end{theorem}

\begin{remark}
  \label{rem:cyclic}
  By Theorem~\ref{thm:gale} the facets depend only on the order of the points
  along the curve, so all cyclic $4$-polytopes with $7$ vertices have one
  combinatorial type \cite[Example~0.6]{ziegler-1995}, which we denote
  $C_{4}(7)$ (in \cite{horvath-langi-2016} $C_{4}(7)$ denotes a particular
  polytope of this type).
  Labelling the vertices $1, \ldots, 7$ along the curve, its $14$ facets
  are~\artref{testprelim}
  \[
    1234,\; 1237,\; 1245,\; 1256,\; 1267,\; 1347,\; 1457,\;
    1567,\; 2345,\; 2356,\; 2367,\; 3456,\; 3467,\; 4567.
  \]
  Of the five combinatorial types of simplicial $4$-polytopes with $7$
  vertices in Firsching's enumeration \cite{firsching-2017}, $C_{4}(7)$ is the
  first, labelled \texttt{firsching-7-0} in our code; the other four are not
  cyclic~\artref{testresults}, \artref{testprelim}.
\end{remark}

For even $d$, Theorem~\ref{thm:hl} leaves the cyclic polytopes open:

\begin{cited}{Problem 2}{{\cite[Section~5]{horvath-langi-2016}}}
  Is it true that if $P \in \mathcal{P}_{d}(d+3)$, where $d$ is even, has
  volume $v_{d}(d+3)$, then $P$ is not cyclic?
\end{cited}

We answer Problem~2 for $d = 4$. By Remark~\ref{rem:cyclic}, every cyclic
$P \in \mathcal{P}_{4}(7)$ has type $C_{4}(7)$, so it suffices to show that
every such polytope has volume less than the $\sqrt{3}/4$ of
Remark~\ref{rem:dim4}. Section~\ref{sec:method} describes how we prove this
bound, and Section~\ref{sec:results} states it.

\subsection{Cross-polytopes and eight vertices}

\begin{definition}
  \label{def:cross}
  A $d$-polytope is a \emph{cross-polytope} if it is combinatorially
  equivalent to $\conv\{\pm e_{1}, \ldots, \pm e_{d}\}$, whose $2^{d}$ facets
  are the convex hulls of the sets containing exactly one of $e_{i}$ and
  $-e_{i}$ for each $i$. A
  cross-polytope $\conv\{\pm u_{1}, \ldots, \pm u_{d}\}$ with $u_{1}, \ldots,
  u_{d}$ an orthonormal basis is \emph{regular}; it lies in
  $\mathcal{P}_{d}(2d)$ and has volume $2^{d}/d!$, which is $2/3$ for
  $d = 4$~\artref{testeight}.
\end{definition}

\begin{theorem}[{\cite[Corollary~3]{horvath-langi-2016}}]
  \label{thm:cross}
  If $P \in \mathcal{P}_{d}(2d)$ has maximal volume in the combinatorial class
  of cross-polytopes inscribed in $\mathbb{S}^{d-1}$, then it is a regular
  cross-polytope.
\end{theorem}

\begin{theorem}[{\cite[Theorem~1]{grunbaum-sreedharan-1967}}]
  \label{thm:gs}
  There are exactly $37$ combinatorial types of simplicial $4$-polytopes with
  $8$ vertices.
\end{theorem}

\begin{remark}
  \label{rem:eighttypes}
  Firsching \cite{firsching-2017} gives a polytope of each type with rational
  vertices on $\mathbb{S}^{3}$; we label them \texttt{firsching-8-0} to
  \texttt{firsching-8-36}. We check in exact arithmetic that these $37$
  polytopes are simplicial and pairwise not combinatorially equivalent, so by
  Theorem~\ref{thm:gs} they represent every type, and that
  \texttt{firsching-8-23} is a cross-polytope~\artref{testeight}.
\end{remark}

\section{Method}
\label{sec:method}

By Section~\ref{sec:prelim} it suffices to bound the volume of every
$P \in \mathcal{P}_{4}(7)$ of type $C_{4}(7)$, and of every
$P \in \mathcal{P}_{4}(8)$ of each simplicial type other than the
cross-polytope. For one type at a time, we write $d!$ times the volume
as a polynomial $D$ in the coordinates of the vertices, bound $D$ on
$(\mathbb{S}^{d-1})^{n}$ by an algebraic identity whose right-hand side is
visibly non-negative there, and check that identity in exact arithmetic.

\subsection{The signed volume}

\begin{definition}
  \label{def:signed-volume}
  An \emph{orientation} of a finite set is an ordering of its elements up to
  even permutations. Let $\mathcal{C} = (n, \mathcal{F})$, where
  $\mathcal{F}$ is a finite set of ordered $d$-tuples
  $F = (F_{1}, \ldots, F_{d})$ of distinct elements of $\{1, \ldots, n\}$,
  called facets. Deleting $F_{a}$ gives a ridge,
  with the orientation of $(F_{1}, \ldots, \widehat{F_{a}}, \ldots, F_{d})$
  if $a$ is even and the opposite one if $a$ is odd. The orientation of
  $\mathcal{C}$ is \emph{coherent} if every ridge lies in exactly two facets
  and these induce opposite orientations on it. A point
  $p = (p_{1}, \ldots, p_{n}) \in (\mathbb{R}^{d})^{n}$ \emph{realises}
  $\mathcal{C}$ if $P = \conv\{p_{1}, \ldots, p_{n}\}$ has the $p_{i}$ as its
  vertices and its facets are exactly the simplices
  $\conv\{p_{j} : j \in F\}$, $F \in \mathcal{F}$. The \emph{signed volume} of
  $p$ is
  \begin{equation}
    \label{eq:signed-volume}
    D(p) \;=\; \sum_{F \in \mathcal{F}}
      \det\bigl(p_{F_{1}}, \ldots, p_{F_{d}}\bigr).
  \end{equation}
\end{definition}

Writing $x_{i,a}$ for the $a$-th coordinate of $p_{i}$ and
$x_{i} = (x_{i,1}, \ldots, x_{i,d})$, $D$ is a polynomial of degree $d$ in the
$dn$ variables $x_{i,a}$ with integer coefficients \artref{signedvolume}.

\begin{lemma}
  \label{lem:volume}
  If the orientation of $\mathcal{C}$ is coherent and $p$ realises
  $\mathcal{C}$, then $d!\,\vol_{d}(P) = \lvert D(p)\rvert$.
\end{lemma}

\begin{proof}
  Fix $y \in \mathbb{R}^{d}$ and expand
  $\det(p_{F_{1}} - y, \ldots, p_{F_{d}} - y)$ by multilinearity. Terms
  containing $y$ twice vanish. A term containing $y$ once is, up to sign, the
  determinant of $y$ and the vertices of a ridge of $F$ in their induced
  order, and by coherence the two facets at that ridge contribute it with
  opposite signs. Hence
  $D(p) = \sum_{F} \det(p_{F_{1}} - y, \ldots, p_{F_{d}} - y)$ for every $y$.
  Take $y$ in the interior of $P$. The simplices $\conv(\{y\} \cup F)$ tile
  $P$, and each summand is $\pm d!$ times the volume of one of them. The sign
  is $+$ exactly when the ordering of $F$ agrees with the orientation that $P$
  induces on its boundary. That orientation is coherent. Any two facets of
  $P$ are joined by a sequence of facets in which consecutive ones share a
  ridge, and along such a sequence a coherent orientation is determined by
  its value on the first facet. Hence the orientation of $\mathcal{C}$ agrees
  with that of the boundary on every facet or on none, and the signs are all
  $+$ or all $-$.
\end{proof}

\begin{remark}[{\cite{berman-hanes-1970}, \cite[Lemma~3.5]{hoehner-ledford-2026}}]
  For $d = 3$, Berman and Hanes state this identity without proof, and Hoehner
  and Ledford prove it with the divergence theorem.
\end{remark}

\subsection{Certificates}

To prove $D \leq \beta$ on $(\mathbb{S}^{d-1})^{n}$, we write $\beta - D$ as a
polynomial that is non-negative on $\mathbb{R}^{dn}$ plus a polynomial that
vanishes on $(\mathbb{S}^{d-1})^{n}$.

\emph{Sums of squares.} A \emph{monomial} is a product of the variables
$x_{i,a}$, repetitions allowed; the empty product $1$ is the monomial of
degree $0$. Let $v(x)$ be the column vector of all monomials of degree at
most $2$, in a fixed order,
\[
  v(x) = \bigl(1,\; x_{1,1}, \ldots, x_{n,d},\; x_{1,1}^{2},\;
  x_{1,1}x_{1,2}, \ldots, x_{n,d}^{2}\bigr)^{\mathsf{T}},
\]
which has $1 + dn + \binom{dn+1}{2}$ entries.

Every polynomial $g$ of degree at most $2$ is $c^{\mathsf{T}} v(x)$ for a
unique vector $c$, and then $g^{2} = v^{\mathsf{T}} (cc^{\mathsf{T}})\, v$.
Hence the sums of squares of such polynomials are exactly the polynomials
\[
  v^{\mathsf{T}} Q\, v = \lvert L^{\mathsf{T}} v\rvert^{2},
  \qquad Q = LL^{\mathsf{T}} \text{ symmetric positive semidefinite}.
\]
Products of two entries of $v$ have degree at most $4$, the degree of $D$.

\emph{Multiples of the sphere equations.} A sum of squares alone cannot
equal $\beta - D$. Since $D(tx) = t^{d} D(x)$, the polynomial $\beta - D$
takes negative values on $\mathbb{R}^{dn}$. The bound holds only because
\[
  \lvert x_{i}\rvert^{2} = x_{i,1}^{2} + \cdots + x_{i,d}^{2} = 1,
  \qquad i = 1, \ldots, n,
\]
so the identity has to use these equations.

For any polynomials $h_{1}, \ldots, h_{n}$, the polynomial
\[
  \sum_{i=1}^{n} h_{i}(x)\bigl(\lvert x_{i}\rvert^{2} - 1\bigr)
\]
vanishes on $(\mathbb{S}^{d-1})^{n}$.

\begin{definition}
  \label{def:certificate}
  A \emph{certificate} for $\mathcal{C}$ and $\beta \in \mathbb{Q}$ is a
  symmetric positive definite matrix $Q$ with rational entries together with
  polynomials $h_{1}, \ldots, h_{n}$ with rational coefficients such that
  \begin{equation}
    \label{eq:certificate}
    \beta - D(x) \;=\; v(x)^{\mathsf{T}} Q\, v(x)
    \;+\; \sum_{i=1}^{n} h_{i}(x)\bigl(\lvert x_{i}\rvert^{2} - 1\bigr)
  \end{equation}
  holds identically as polynomials in the $x_{i,a}$ \artref{certificate}.
\end{definition}

Because \eqref{eq:certificate} is an identity of polynomials with rational
data, it can be verified in exact arithmetic by comparing coefficients
(Section~\ref{sec:verification}).

\begin{proposition}
  \label{prop:certificate}
  Suppose the orientation of $\mathcal{C}$ is coherent and a certificate for
  $\mathcal{C}$ and $\beta$ exists. Then $\lvert D(p)\rvert \leq \beta$ for
  every $p \in (\mathbb{S}^{d-1})^{n}$, and every polytope inscribed in
  $\mathbb{S}^{d-1}$ that realises $\mathcal{C}$ has volume at most
  $\beta/d!$.
\end{proposition}

\begin{proof}
  Let $p \in (\mathbb{S}^{d-1})^{n}$. In \eqref{eq:certificate} the second sum
  vanishes at $p$ and the first is non-negative, so $D(p) \leq \beta$. Let
  $\rho$ be a reflection of $\mathbb{R}^{d}$. Then
  $\rho p = (\rho p_{1}, \ldots, \rho p_{n})$ also lies in
  $(\mathbb{S}^{d-1})^{n}$, and $D(\rho p) = \det(\rho)\,D(p) = -D(p)$, so
  $-D(p) \leq \beta$ as well. Hence $\lvert D(p)\rvert \leq \beta$, and
  Lemma~\ref{lem:volume} gives the volume bound.
\end{proof}

\begin{remark}[{\cite{parrilo-2003}, \cite{putinar-1993}}]
  Identities of the form \eqref{eq:certificate} are the sum-of-squares
  certificates of Parrilo. Their existence is not automatic. If $\beta$ is
  below the maximum of $D$ on $(\mathbb{S}^{d-1})^{n}$ there is none, by
  Proposition~\ref{prop:certificate}. If $\beta$ is above it, Putinar's
  theorem gives an identity of this form with $Q$ positive semidefinite, but
  with monomials of some degree that it does not bound. We use monomials of
  degree at most $2$, the smallest degree for which $v^{\mathsf{T}} Q\, v$
  reaches the degree $4$ of $D$, and nothing guarantees a certificate of this
  degree in advance: the ones used here exist because they were found and
  checked. Positive semidefiniteness of $Q$ would suffice, but we require
  positive definiteness because, by Sylvester's criterion, it can be checked
  from the leading principal minors alone, whereas positive semidefiniteness
  needs all principal minors.
\end{remark}

\subsection{Verification}
\label{sec:verification}

A certificate is finite data: the ordered facets of $\mathcal{C}$, $\beta$,
the entries of $Q$ and the coefficients of the $h_{i}$, all rational. The
verifier checks the hypotheses of Proposition~\ref{prop:certificate} for this
data in exact rational arithmetic \artref{verifier}:
\begin{enumerate}
  \item \emph{Identity.} $D$ is rebuilt from the stored facets by
    \eqref{eq:signed-volume}, both sides of \eqref{eq:certificate} are
    expanded, and every coefficient must agree.
  \item \emph{Positive definiteness.} By Sylvester's criterion, $Q$ is
    positive definite if and only if every leading principal minor is
    positive. The minors are computed exactly by the fraction-free
    elimination of Bareiss \cite{bareiss-1968}.
  \item \emph{Orientation and type.} Every ridge must lie in exactly two facets with
    opposite induced orientations, so the orientation is coherent. The facet
    vertex sets must also be exactly those of the type: for $C_{4}(7)$ the
    fourteen facets listed in Remark~\ref{rem:cyclic}, which the verifier
    rebuilds from Gale's condition, and for eight vertices those of the
    polytope of Remark~\ref{rem:eighttypes} \artref{verifier}. Without these
    checks the bound would still hold for $D$, but need not bound the volume
    of a polytope of the type.
\end{enumerate}
The stored certificate for $C_{4}(7)$ passes all three checks and has
$\beta = 48/5$ \artref{certificates}, \artref{testmethod}; so do the $36$
certificates for eight vertices, with $\beta = 78/5$
\artref{certificateseight}.

\subsection{How the certificates were found}

\begin{remark}
  The certificates were found numerically and then made exact. The proofs use
  none of the steps below, only the checks of
  Section~\ref{sec:verification} \artref{certscript}.
\end{remark}

\begin{remark}
  \label{rem:computation}
  The entries of $Q$ and the coefficients of the $h_{i}$ enter
  \eqref{eq:certificate} linearly, and the only other condition is that $Q$
  be positive semidefinite, so finding them is a semidefinite program. We
  solved it in floating point with Clarabel \cite{goulart-chen-2026} through
  CVXPY \cite{diamond-boyd-2016}, making the smallest eigenvalue of $Q$ as
  large as possible \artref{search}.
\end{remark}

\begin{remark}
  Changing the signs of two coordinates of every point, say $x_{i,a}$ and
  $x_{i,b}$ for all $i$, leaves $D$ unchanged. We grouped the monomials by how
  they change sign under these maps and took $Q$ block diagonal, with one
  block per group \artref{basis}. This made the program much smaller; the
  verifier checks each block.
\end{remark}

\begin{remark}
  We rounded every number to a rational with denominator $10^{9}$ and added
  the small leftover of \eqref{eq:certificate} to matching entries of $Q$, so
  that \eqref{eq:certificate} holds exactly. This is a simpler form of the
  rounding of Peyrl and Parrilo \cite{peyrl-parrilo-2008} \artref{rounding}.
  The verifier checks that $Q$ is still positive definite.
\end{remark}

\begin{remark}
  For $C_{4}(7)$, any $\beta < 4! \cdot \sqrt{3}/4 \approx 10.39$ with a
  certificate proves Corollary~\ref{cor:problem2}, and none exists below the
  maximum of $D$, which is at least the value of about $9.34$ found
  numerically \artref{ascent}, \artref{certificates}, \artref{testmethod}. We chose $\beta = 48/5$, giving the simple
  bound $2/5$, and did not look for the smallest possible $\beta$. For eight
  vertices, any common $\beta < 4! \cdot 2/3 = 16$ suffices for
  Corollary~\ref{cor:eight}. We chose $\beta = 78/5$, giving the bound
  $13/20$, above the largest numerical value of $D$ among the $36$ types,
  about $4! \cdot 0.601 \approx 14.4$ \artref{sweepeight}.
\end{remark}

\section{Seven vertices}
\label{sec:results}

\begin{theorem}
  \label{thm:main}
  Every $P \in \mathcal{P}_{4}(7)$ of combinatorial type $C_{4}(7)$ has
  $\vol_{4}(P) \leq 2/5$.
\end{theorem}

\begin{proof}
  The stored certificate for $C_{4}(7)$ passes the checks of
  Section~\ref{sec:verification} with $\beta = 48/5$
  \artref{certificates}, \artref{testresults}. After relabelling its
  vertices, $P$ realises the coherently oriented sphere of that certificate,
  so Proposition~\ref{prop:certificate} gives
  $\vol_{4}(P) \leq \beta/4! = 2/5$.
\end{proof}

\begin{corollary}
  \label{cor:problem2}
  If $P \in \mathcal{P}_{4}(7)$ has volume $v_{4}(7)$, then $P$ is not
  cyclic. This answers Problem~2 for $d = 4$.
\end{corollary}

\begin{proof}
  By Remark~\ref{rem:dim4}, $v_{4}(7) \geq \sqrt{3}/4 > 2/5$
  \artref{testresultssec}. By Remark~\ref{rem:cyclic} a cyclic
  $P \in \mathcal{P}_{4}(7)$ has type $C_{4}(7)$, so by
  Theorem~\ref{thm:main} it has volume at most $2/5 < v_{4}(7)$.
\end{proof}

\begin{corollary}
  \label{cor:seven}
  $v_{4}(7) = \sqrt{3}/4$.
\end{corollary}

\begin{proof}
  Let $P \in \mathcal{P}_{4}(7)$ have volume $v_{4}(7)$; such $P$ exist by
  Remark~\ref{rem:attained}. By Remark~\ref{rem:propertyz}, $P$ satisfies
  Property~Z, and by Corollary~\ref{cor:problem2} it is not cyclic. Hence $P$
  has maximal volume over the non-cyclic elements of $\mathcal{P}_{4}(7)$
  satisfying Property~Z, and Theorem~\ref{thm:hl} with
  Remark~\ref{rem:dim4} gives $\vol_{4}(P) = \sqrt{3}/4$.
\end{proof}

\begin{remark}
  The proof of Theorem~\ref{thm:main} is computer-assisted: it relies on the
  exact verifier \artref{verifier} and on the stored certificate, not on any
  floating-point computation. We do not claim that $2/5$ is sharp; the
  largest volume of type $C_{4}(7)$ found numerically is about $0.3891$
  \artref{ascent}, \artref{testresultssec}. Problem~2 remains open for even
  $d \geq 6$.
\end{remark}

\section{Eight vertices}
\label{sec:eight}

\begin{theorem}
  \label{thm:eight}
  Every simplicial $P \in \mathcal{P}_{4}(8)$ that is not a cross-polytope has
  $\vol_{4}(P) \leq 13/20$.
\end{theorem}

\begin{proof}
  By Remark~\ref{rem:eighttypes}, $P$ has the type of one of the $36$
  polytopes \texttt{firsching-8-$k$} with $k \neq 23$. For each of them the
  stored certificate passes the checks of Section~\ref{sec:verification} with
  $\beta = 78/5$ \artref{certificateseight}, \artref{testeight}. After
  relabelling its vertices, $P$ realises the coherently oriented sphere of
  that certificate, so Proposition~\ref{prop:certificate} gives
  $\vol_{4}(P) \leq \beta/4! = 13/20$.
\end{proof}

\begin{corollary}
  \label{cor:eight}
  $v_{4}(8) = 2/3$, and every $P \in \mathcal{P}_{4}(8)$ with volume
  $v_{4}(8)$ is a regular cross-polytope.
\end{corollary}

\begin{proof}
  By Definition~\ref{def:cross}, $v_{4}(8) \geq 2/3 > 13/20$. Let
  $P \in \mathcal{P}_{4}(8)$ have volume $v_{4}(8)$; such $P$ exist by
  Remark~\ref{rem:attained}. By Remark~\ref{rem:propertyz} and
  Lemma~\ref{lem:simplicial}, $P$ is simplicial, so by
  Theorem~\ref{thm:eight} it is a cross-polytope. It has maximal volume among
  all of $\mathcal{P}_{4}(8)$, hence among the cross-polytopes in it, and
  Theorem~\ref{thm:cross} shows that it is regular, of volume $2/3$.
\end{proof}

\begin{remark}
  As for Theorem~\ref{thm:main}, the proof relies on the exact verifier
  \artref{verifier} and the stored certificates \artref{certificateseight}. We do not claim that $13/20$ is sharp; among the
  $36$ types, the largest value of $D/4!$ found numerically is about $0.601$,
  for \texttt{firsching-8-22} \artref{sweepeight}.
\end{remark}

\section*{AI assistance}

An AI assistant was used during this work, including for mathematical
discussion, implementation of the code, drafting and revising the text, and
checking the cited statements against their sources. The approach is the
author's, and the author checked the arguments and the references. The proofs
do not depend on the assistant: they rest on the cited results and on the
certificates, which the exact verifier checks.

\section*{Code and data}
\label{sec:artifacts}

{\footnotesize
Everything reported here is available at \repohref. The entries below are cited
as [R$k$]; they are artifacts rather than literature and are numbered separately
from the bibliography.

\begin{artifactlist}
  \artifact{certificates}{results/certificates/d4-n07}{the exact certificates;
    the one used here is \texttt{cyclic-4-7}, the others are for three of the
    non-cyclic types; \texttt{cyclic-4-7.relaxation.json} stores the numerical
    maximum of $D$ for $C_{4}(7)$; tests in
    \artpath{tests/sos/test\_results.py}}
  \artifact{signedvolume}{src/inscribed/certify/polynomial.py}{the signed volume
    as an exact polynomial; tests in \artpath{tests/sos/test\_polynomial.py}}
  \artifact{certificate}{src/inscribed/certify/certificate.py}{the stored data of a
    certificate; tests in \artpath{tests/sos/test\_certificate.py}}
  \artifact{verifier}{src/inscribed/certify/verify.py}{the exact verifier of
    Section~\ref{sec:verification}, with
    \artpath{src/inscribed/combinatorics/complex.py} checking coherence,
    \artpath{src/inscribed/certify/cyclic.py} rebuilding the facets of
    $C_{4}(7)$ from Gale's condition, and \artpath{src/inscribed/sos/types.py}
    with \artpath{src/inscribed/enumeration/firsching.py} those of the
    polytopes of Remark~\ref{rem:eighttypes}, from Firsching's files fetched by
    \artpath{scripts/setup\_data.py}; tests in
    \artpath{tests/sos/test\_verify.py}}
  \artifact{testresults}{tests/sos/test\_results.py}{tests re-verifying the
    committed certificates and identifying the certified type with $C_{4}(7)$}
  \artifact{certscript}{scripts/sos\_certificate.py}{searches, rounds, verifies
    and stores a certificate, and re-verifies a stored one}
  \artifact{basis}{src/inscribed/sos/basis.py}{the parity blocks; tests in
    \artpath{tests/sos/test\_basis.py}}
  \artifact{search}{src/inscribed/sos/search.py}{the semidefinite program of
    Remark~\ref{rem:computation}; tests in \artpath{tests/sos/test\_pipeline.py}}
  \artifact{rounding}{src/inscribed/sos/rounding.py}{rounding and repair; tests
    in \artpath{tests/sos/test\_rounding.py}}
  \artifact{ascent}{src/inscribed/search/ascent.py}{block-coordinate ascent on
    $D$, run by \artpath{scripts/sos\_bound.py}; tests in
    \artpath{tests/test\_search.py}}
  \artifact{testprelim}{tests/papers/maxvol\_d4/test\_preliminaries.py}{tests
    backing the computations of Section~\ref{sec:prelim}}
  \artifact{testmethod}{tests/papers/maxvol\_d4/test\_method.py}{tests backing the
    claims of Section~\ref{sec:method}}
  \artifact{testresultssec}{tests/papers/maxvol\_d4/test\_seven.py}{tests
    backing the numbers of Section~\ref{sec:results}}
  \artifact{certificateseight}{results/certificates/d4-n08}{the exact
    certificates for the $36$ types of Theorem~\ref{thm:eight}}
  \artifact{sweepeight}{results/d4-n08.json}{the numerical maxima of $D/4!$
    for the $37$ types with eight vertices}
  \artifact{testeight}{tests/papers/maxvol\_d4/test\_eight.py}{tests backing
    Remark~\ref{rem:eighttypes}, Definition~\ref{def:cross} and
    Section~\ref{sec:eight}, including the exact facet and equivalence checks}
\end{artifactlist}
}

\bibliographystyle{amsplain}
\bibliography{references}

\end{document}